\documentclass[a4paper, 11pt
]{article} 

\usepackage[latin1]{inputenc}
\usepackage{subfigure}
\usepackage{amsmath}
\usepackage{amsfonts}
\usepackage{amssymb}
\usepackage{amsthm}
\usepackage{multicol}
\usepackage{array}
\usepackage{tikz}
\usetikzlibrary{intersections}
\usetikzlibrary{calc}
\usepackage{fullpage}

\newtheorem{lem}{Lemma}[section]
\newtheorem{coro}{Corollary}[section]
\newtheorem{prop}{Proposition}[section]
\newtheorem{thm}{Theorem}[section]

\newcommand{\sgn}{\operatorname{sgn}}
\newcommand{\complexes}{\mathbb{C}}
\newcommand{\integers}{\mathbb{Z}}
\newcommand{\naturals}{\mathbb{N}}

\author{Wenjie Fang \footnote{This work is partially supported by ANR IComb (ANR-08-JCJC-0011) and ANR Cartaplus (ANR-12-JS02-001-01).} \\ LIAFA, Universit\'e Paris Diderot - Paris 7 \\ B\^atiment Sophie Germain, 75205 Paris Cedex 13, France}

\title{A generalization of the quadrangulation relation to constellations and hypermaps}

\begin{document}
\maketitle
\begin{abstract}
Constellations and hypermaps generalize combinatorial maps, \textit{i.e.} embedding of graphs in a surface, in terms of factorization of permutations. In this paper, we extend a result of Jackson and Visentin (1990) stating an enumerative relation between quadrangulations and bipartite quadrangulations. We show a similar relation between hypermaps and constellations by using a result of Littlewood on factorization of characters. A combinatorial proof of Littlewood's result is also given. Furthermore, we show that coefficients in our relation are all positive integers, hinting possibility of a combinatorial interpretation. Using this enumerative relation, we recover a result on the asymptotic behavior of hypermaps in Chapuy (2009).
\end{abstract}

\section{Introduction} \label{sec:intro}
Maps are combinatorial structures describing an embedding of a graph in a surface. They can be encoded as factorizations of the identity element in the symmetric group. Enormous efforts have been devoted to the enumeration of these combinatorial objects and their variants, see \textit{e.g.} \cite{lando2004graphs, bousquet2000enumeration, bouttier2004planar} and references therein. In \cite{jackson1999combinatorial}, the following strikingly simple enumerative relation was established:
\begin{displaymath}
E^{(g)}_{n,D} = \sum_{i=0}^{g} 4^{g-i} B^{(g-i, 2i)}_{n,D} = 4^g B^{(g,0)}_{n,D} + 4^{g-1} B^{(g-1, 2)}_{n,D} + \ldots + B^{(0,2g)}_{n,D} .
\end{displaymath}
Here, for $D \subseteq \naturals^{+}$, we define $B^{(g,k)}_{n,D}$ as the number of rooted bipartite maps of genus $g$ with every face degree of the form $2d$ and $d \in D$, whose vertices are colored black and white, rooted in a white vertex and with $n$ edges such that $k$ black vertices are marked. The number $E^{(g)}_{n,D}$ is the counterpart for rooted (non necessarily bipartite) maps with the same restriction on face degrees without marking. In the planar case, we have $E^{(0)}_{n,D} = B^{(0,0)}_{n,D}$, meaning that a planar map with all faces of even degree is always bipartite. The situation in higher genera is more complicated, and every map whose faces are all of even degree is not always bipartite. Figure \ref{fig:grid-torus} gives such an example of a $5 \times 6$ rectangular grid on a torus, which is a quadrangulation but not bipartite.

\begin{figure}
\begin{center}
\begin{tikzpicture}[scale=0.75]
\clip (-2.8, -2.3) rectangle (10, 2.3);
\begin{scope}
\foreach \x in {-2.4, -1.6, ..., 2.4} \draw (\x, -2) -- (\x, 2);
\foreach \y in {-2, -1.2, ..., 2.1} \draw (-2.4, \y) -- (2.4, \y);
\foreach \x in {-2.4, -1.6, ..., 2.4} \foreach \y in {-2, -1.2, ..., 2.1} \fill[black] (\x, \y) circle (2pt);
\draw[dashed, latex-latex] (-2.4, 0) .. controls +(135:2) and +(45:2) .. (2.4, 0);
\draw[dashed, latex-latex] (0.4, -2) .. controls +(-45:1.7) and +(45:1.7) .. (0.4, 2);
\draw[very thick, -latex] (2.8,0) -- (4.4,0);
\end{scope}
\begin{scope}[xshift=7cm]
\draw[name path = c1] (0,0) circle (1cm);
\draw[name path = c2, densely dotted] (0,0) circle (1.3cm);
\draw[name path = c3] (0,0) circle (1.6cm);
\draw[name path = c4, densely dotted] (0,0) circle (1.9cm);
\draw[name path = c5] (0,0) circle (2.2cm);
\draw[name path = v1] (0:1) .. controls +(90:0.5) and +(90:0.5) .. (0:2.2);
\draw[name path = t1, densely dotted] (0:1) .. controls +(-90:0.5) and +(-90:0.5) .. (0:2.2);
\draw[name path = v2] (60:1) .. controls +(150:0.5) and +(150:0.5) .. (60:2.2);
\draw[name path = t2, densely dotted] (60:1) .. controls +(-30:0.5) and +(-30:0.5) .. (60:2.2);
\draw[name path = v3] (120:1) .. controls +(210:0.5) and +(210:0.5) .. (120:2.2);
\draw[name path = t3, densely dotted] (120:1) .. controls +(30:0.5) and +(30:0.5) .. (120:2.2);
\draw[name path = v4] (180:1) .. controls +(270:0.5) and +(270:0.5) .. (180:2.2);
\draw[name path = t4, densely dotted] (180:1) .. controls +(90:0.5) and +(90:0.5) .. (180:2.2);
\draw[name path = v5] (240:1) .. controls +(-30:0.5) and +(-30:0.5) .. (240:2.2);
\draw[name path = t5, densely dotted] (240:1) .. controls +(150:0.5) and +(150:0.5) .. (240:2.2);
\draw[name path = v6] (300:1) .. controls +(30:0.5) and +(30:0.5) .. (300:2.2);
\draw[name path = t6, densely dotted] (300:1) .. controls +(210:0.5) and +(210:0.5) .. (300:2.2);

\fill[black, name intersections={of= c1 and v1}] (intersection-1) circle (2pt);
\fill[black, name intersections={of= c3 and v1}] (intersection-1) circle (2pt);
\fill[black, name intersections={of= c5 and v1}] (intersection-1) circle (2pt);
\fill[black, name intersections={of= c1 and v2}] (intersection-1) circle (2pt);
\fill[black, name intersections={of= c3 and v2}] (intersection-1) circle (2pt);
\fill[black, name intersections={of= c5 and v2}] (intersection-1) circle (2pt);
\fill[black, name intersections={of= c1 and v3}] (intersection-1) circle (2pt);
\fill[black, name intersections={of= c3 and v3}] (intersection-1) circle (2pt);
\fill[black, name intersections={of= c5 and v3}] (intersection-1) circle (2pt);
\fill[black, name intersections={of= c1 and v4}] (intersection-1) circle (2pt);
\fill[black, name intersections={of= c3 and v4}] (intersection-1) circle (2pt);
\fill[black, name intersections={of= c5 and v4}] (intersection-1) circle (2pt);
\fill[black, name intersections={of= c1 and v5}] (intersection-1) circle (2pt);
\fill[black, name intersections={of= c3 and v5}] (intersection-1) circle (2pt);
\fill[black, name intersections={of= c5 and v5}] (intersection-1) circle (2pt);
\fill[black, name intersections={of= c1 and v6}] (intersection-1) circle (2pt);
\fill[black, name intersections={of= c3 and v6}] (intersection-1) circle (2pt);
\fill[black, name intersections={of= c5 and v6}] (intersection-1) circle (2pt);

\fill[black, name intersections={of= c2 and t1}] (intersection-1) circle (2pt);
\fill[black, name intersections={of= c4 and t1}] (intersection-1) circle (2pt);
\fill[black, name intersections={of= c2 and t2}] (intersection-1) circle (2pt);
\fill[black, name intersections={of= c4 and t2}] (intersection-1) circle (2pt);
\fill[black, name intersections={of= c2 and t3}] (intersection-1) circle (2pt);
\fill[black, name intersections={of= c4 and t3}] (intersection-1) circle (2pt);
\fill[black, name intersections={of= c2 and t4}] (intersection-1) circle (2pt);
\fill[black, name intersections={of= c4 and t4}] (intersection-1) circle (2pt);
\fill[black, name intersections={of= c2 and t5}] (intersection-1) circle (2pt);
\fill[black, name intersections={of= c4 and t5}] (intersection-1) circle (2pt);
\fill[black, name intersections={of= c2 and t6}] (intersection-1) circle (2pt);
\fill[black, name intersections={of= c4 and t6}] (intersection-1) circle (2pt);
\end{scope}
\end{tikzpicture}
\end{center}
\caption{An example of a $5 \times 6$ grid on a torus} \label{fig:grid-torus}
\end{figure}
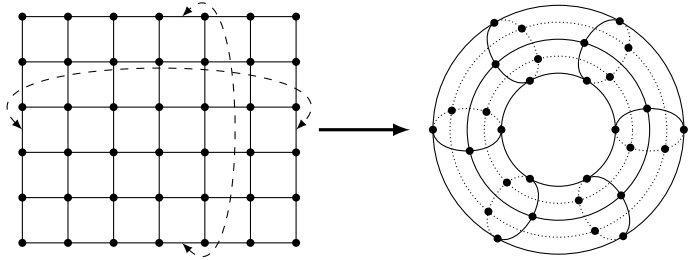

The special case on $D=\{ 2 \}$ had been proved in \cite{jackson1990character}, and the maps in concern are quadrangulations, which gives this special case the name \emph{quadrangulation relation}. It had been then extended to $D= \{ p \}$ in \cite{jackson1990character-1}. Despite its nice form, the combinatorial meaning of the quadrangulation relation remains unclear, though some effort is done in \cite{jackson1999combinatorial} to explore properties of the possible hinted bijection.

In enumeration of maps, there is a recurrent phenomenon: results on bipartite maps can often be generalized to constellations (see \textit{e.g.} \cite{bousquet2000enumeration, bouttier2004planar, poulalhon2002factorizations}). In fact, since bipartite maps are exactly $2$-constellations, constellations can be seen as a generalization of bipartite maps. In the same spirit, we will generalize the quadrangulation relation to $m$-constellations and $m$-hypermaps. See Section \ref{sec:defs-maps} for the definitions of these notions. As an example, our result in the case $m=3$ gives rise to the following relation (\textit{c.f.} Corollary \ref{coro:counting-relation-m-3-4}):
\begin{displaymath}
H^{(g)}_{n,3,D} = \sum_{i=0}^{g} 3^{2g-2i} \sum_{l=0}^{2i} \frac{2^{l+1} - (-1)^{l+1}}{3} C^{(g-i, l, 2i-l)}_{n,3,D}.
\end{displaymath}
Here, $C^{(g, a, b)}_{n,3,D}$ is the number of rooted 3-constellations with $n$ hyperedges, and hyperface degree restricted by the set $D$, with $a$ marked vertices of color 1 and $b$ marked vertices of color 2. The number $H^{(g)}_{n,3,D}$ is the counterpart for rooted 3-hypermaps without marking. We also give the same type of relation for general $m$. While our generalization of the quadrangulation relation still has a simple form, it involves extra coefficients in the weighted sum. Explicit expressions of these coefficients are given in Corollary \ref{coro:explicit-coeffs} using symmetries in $m$-constellations. We then establish Theorem \ref{thm:positivity-of-differential-operator-coefficient} stating that these coefficients are all positive integers, revealing the possibility that a potential combinatorial interpretation exists for our relation. Finally, we recover a relation between the asymptotic behavior of $m$-constellations and $m$-hypermaps in \cite{chapuy2009asymptotic}, which can be seen as an asymptotic version of our relation.

Given a partition $\mu \vdash n$, we note $m\mu$ the partition obtained by multiplying every part in $\mu$ by $m$. In \cite{jackson1990character}, the quadrangulation relation was obtained using a factorization of irreducible characters of the symmetric group on partitions of the form $[(mk)^n]$ using a notion called $m$-balanced partition, which is a special case of a more general result for characters evaluated on partitions of the form $m\mu$ in an article of Littlewood \cite{littlewood1951modular}. For the sake of self-containedness, a combinatorial proof of this result is given here.

\section{Preliminaries} \label{sec:defs}

\subsection{Constellations and hypermaps} \label{sec:defs-maps}
A \emph{map} $M$ is an embedding of a connected graph $G$, with possibly multi-edges or loops, into a closed, connected and oriented surface $S$ such that all \emph{faces}, \textit{i.e.} components of $S \setminus M$, are topological disks. Maps are defined up to orientation-preserving homeomorphisms. We define the genus $g$ of a map to be that of the surface it is embedded into. We thus have the Euler relation $|V|-|E|+|F| = 2 - 2g$.

Following \cite{lando2004graphs, chapuy2009asymptotic}, we now define two special kinds of maps. An \emph{$m$-hypermap} is a map with two types of faces, \emph{hyperedges} with degree $m$ and \emph{hyperfaces} with degree divisible by $m$, such that every edge is located between a hyperedge and a hyperface. Each edge is then naturally oriented with the hyperedge on its right. Conventionally hyperedges are colored black and hyperfaces white. An \emph{$m$-constellation} is an $m$-hypermap with the additional condition that all vertices are colored with an integer between $1$ and $m$ in a fashion that every hyperedge has its vertices colored by $1, 2, \ldots, m$ in clockwise order. A map with faces of even degree can be considered as a $2$-hypermap by replacing every edge with a $2$-hyperedge, and a bipartite map can be considered as a $2$-constellation in the same way. A \emph{rooted} $m$-hypermap is an $m$-hypermap with a distinguished edge. Rooted $m$-constellations are similarly defined, with the convention that the starting vertex of the root in natural orientation has color $1$. We consider only rooted $m$-hypermaps and rooted $m$-constellations hereinafter. 

Figure \ref{fig:hypermap} provides an example of planar 3-hypermap, which is also a planar 3-constellation with the given vertex coloring. More generally, every planar $m$-hypermap can have its vertices colored to become an $m$-constellation, that is to say, every planar $m$-hypermap can be considered as an $m$-constellation. This is because the only obstacle for an $m$-hypermap to be an $m$-constellation is the existence of a cycle whose length is not divisible by $m$, but in the planar case we can easily show that such a cycle does not exist using the Jordan curve theorem. However, this is not necessarily true for higher genera, in which an $m$-hypermap does not necessarily have a coloring that conforms with the additional condition to be an $m$-constellation. An example for $m=2$ is presented in Figure \ref{fig:grid-torus}, and similar examples can be constructed for general $m$.

\begin{figure}[!htbp]
\begin{center}
\begin{tikzpicture}
\path (0,0) coordinate (a);
\path (0.5,1) coordinate (b);
\path (1,0.3) coordinate (c);
\path (-1,-1) coordinate (d);
\path (0,-2) coordinate (e);
\path (0.6,-1) coordinate (f);
\path (1,-0.7) coordinate (g);
\path (2,-1) coordinate (h);
\path (2,0) coordinate (i);
\path (2,1) coordinate (j);
\path (3,1) coordinate (k);
\path (1.5,1) coordinate (l);
\clip (-1.1, -2.1) rectangle (4.05, 1.55);
\filldraw[fill=gray!50, draw=black] (a) -- (b) -- (c) -- cycle;
\filldraw[fill=gray!50, draw=black] (a) -- (d) -- (e) -- cycle;
\filldraw[fill=gray!50, draw=black] (e) -- (f) -- (h) -- cycle;
\filldraw[fill=gray!50, draw=black] (c) -- (g) -- (h) -- cycle;
\filldraw[fill=gray!50, draw=black] (c) -- (l) -- (i) -- cycle;
\filldraw[fill=gray!50, draw=black] (k) -- (j) -- (i) -- cycle;
\filldraw[fill=gray!50, draw=black] (a) .. controls +(120:1.2) and +(165:1.5) .. (j) .. controls +(20:2) and +(30:2) .. (i) .. controls +(20:7) and +(160:4) .. (a);
\draw[thick] (a) -- (d);
\draw[thick, -latex] (d) -- +(45:0.8);
\end{tikzpicture}
\begin{tikzpicture}
\path (0,0) coordinate (a);
\path (0.5,1) coordinate (b);
\path (1,0.3) coordinate (c);
\path (-1,-1) coordinate (d);
\path (0,-2) coordinate (e);
\path (0.6,-1) coordinate (f);
\path (1,-0.7) coordinate (g);
\path (2,-1) coordinate (h);
\path (2,0) coordinate (i);
\path (2,1) coordinate (j);
\path (3,1) coordinate (k);
\path (1.5,1) coordinate (l);
\clip (-1.1, -2.15) rectangle (6.55, 1.55);
\filldraw[fill=gray!50, draw=black] (a) -- (b) -- (c) -- cycle;
\filldraw[fill=gray!50, draw=black] (a) -- (d) -- (e) -- cycle;
\filldraw[fill=gray!50, draw=black] (e) -- (f) -- (h) -- cycle;
\filldraw[fill=gray!50, draw=black] (c) -- (g) -- (h) -- cycle;
\filldraw[fill=gray!50, draw=black] (c) -- (l) -- (i) -- cycle;
\filldraw[fill=gray!50, draw=black] (k) -- (j) -- (i) -- cycle;
\filldraw[fill=gray!50, draw=black] (a) .. controls +(120:1.2) and +(165:1.5) .. (j) .. controls +(20:2) and +(30:2) .. (i) .. controls +(20:7) and +(160:4) .. (a);
\draw[thick] (a) -- (d);
\draw[thick, -latex] (d) -- +(45:0.8);
\filldraw[draw=black, fill=white] (a) circle (0.1);
\filldraw[draw=black, fill=white] (0.4, 0.9) rectangle(0.6, 1.1) ;
\filldraw[draw=black, fill=black] (c) circle (0.1);
\filldraw[draw=black, fill=black] (d) circle (0.1);
\filldraw[draw=black, fill=white] (-0.1, -2.1) rectangle (0.1, -1.9);
\filldraw[draw=black, fill=black] (f) circle (0.1);
\filldraw[draw=black, fill=white] (0.9, -0.8) rectangle (1.1, -0.6);
\filldraw[draw=black, fill=white] (h) circle (0.1);
\filldraw[draw=black, fill=white] (1.9,-0.1) rectangle (2.1, 0.1);
\filldraw[draw=black, fill=black] (j) circle (0.1);
\filldraw[draw=black, fill=white] (k) circle (0.1);
\filldraw[draw=black, fill=white] (l) circle (0.1);

\node[shape=circle,draw,inner sep=0.3pt] (he1) at (barycentric cs:a=1,b=1,c=1) {\scriptsize{1}};
\node[shape=circle,draw,inner sep=0.3pt] (he2) at (barycentric cs:a=1,d=1,e=1) {\scriptsize{2}};
\node[shape=circle,draw,inner sep=0.3pt] (he3) at (barycentric cs:e=1,f=1,h=1) {\scriptsize{3}};
\node[shape=circle,draw,inner sep=0.3pt] (he4) at (barycentric cs:c=1,g=1,h=1) {\scriptsize{4}};
\node[shape=circle,draw,inner sep=0.3pt] (he5) at (barycentric cs:c=1,l=1,i=1) {\scriptsize{5}};
\node[shape=circle,draw,inner sep=0.3pt] (he6) at (barycentric cs:k=1,j=1,i=1) {\scriptsize{6}};
\node[shape=circle,draw,inner sep=0.3pt] (he7) at (2,1.2) {\scriptsize{7}};

\node at (4.5,-0.5) {$\bullet$ \scriptsize{$ \sigma_1 = (1,5,4)(2)(3)(6,7) $}};
\node at (4.5,-0.9) {$\circ$ \scriptsize{$ \sigma_2 = (1,2,7)(3,4)(5)(6) $}};
\node at (4.5,-1.3) {\scriptsize{$ \Box \; \sigma_3 = (1)(2,3)(4)(5,6,7) $}};
\node[black] at (4.5,-1.7) {\scriptsize{$ \phi = (1,6)(2,5)(3,4)(7) $}};
\end{tikzpicture}
\end{center}
\caption{An example of rooted planar 3-hypermap} \label{fig:hypermap}
\end{figure}
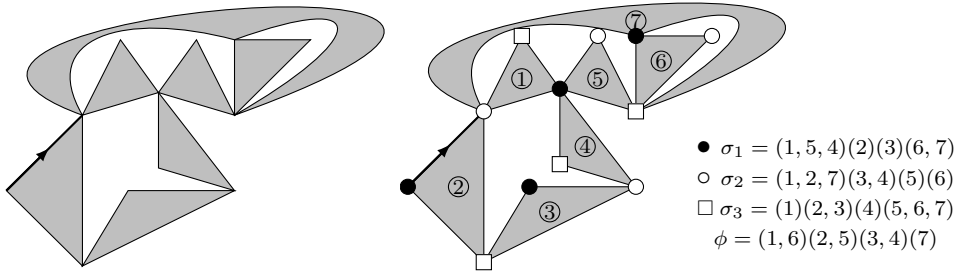

We now define the generating series of $m$-hypermaps and $m$-constellations. We use $\underline{\mathbf{x}}$ to denote a sequence of variables $x_1, \ldots, x_m$, and $[x_i \gets f(i)]$  to denote the substitution of $\underline{\mathbf{x}}$ by $x_i = f(i)$. We also introduce an infinite set of variables $\underline{\mathbf{y}} = y_1, y_2, \ldots$. We define $H(x, \underline{\mathbf{y}}, z, u)$ to be the ordinary generating series of rooted $m$-hypermaps, with $x$ marking the number of vertices, $y_i$ the number of hyperfaces of degree $mi$ for each $i$, $z$ the number of hyperedges and $u$ twice the genus. Similarly, we define $C(\underline{\mathbf{x}}, \underline{\mathbf{y}}, z, u)$ to be the ordinary generating series of rooted $m$-constellations, except that with $x_i$ we mark the number of vertices with color~$i$.

In the following, we take the convention of composing permutations from left to right, considering them as elements of the symmetric group.

Let $S_n$ be the symmetric group consisting of permutations of $n$ elements. A \emph{$k$-factorization of the identity} (or simply \emph{$k$-factorization}) in $S_n$ is a  $k$-tuple of permutations $(\sigma_1, \ldots, \sigma_k)$ in $S_n$ such that $\sigma_1 \cdots \sigma_k = id$. Such a factorization is \emph{transitive} if the family acts transitively on $\{1, \ldots, n\}$. There is a $1$-to-$(n-1)!$ correspondence between rooted $m$-constellations with $n$ hyperedges and transitive $(m+1)$-factorizations in $S_n$ (\textit{c.f.} \cite{lando2004graphs}). More precisely, we first number the hyperedges from $1$ to $n$, with the convention that the rooted hyperedge receives the number $1$. There are $(n-1)!$ ways for such numbering. Then for each color $i$, we construct the permutation $\sigma_i$ whose cycles record hyperedges incident to each vertex of color $i$, in clockwise order. For hyperfaces, we construct the permutation $\phi$ whose cycles record hyperedges incident to each hyperface with an edge from color $m$ to color $1$, in clockwise order. We can simply verify that $\sigma_1 \cdots \sigma_m \phi = id$, which is an $(m+1)$-factorization. Figure~\ref{fig:hypermap} provides an example of such a factorization on a $3$-constellation. Since  the constellation is connected, the $(m+1)$-factorization obtained is always transitive. Each numbering of hyperedges produces one such factorization, hence the $1$-to-$(n-1)!$ correspondence.

Similarly, by numbering corners of faces instead of hyperedges, with a similar construction, we have a $1$-to-$(n-1)! m^{n-1}$ correspondence between rooted $m$-hypermaps with $n$ hyperedges and transitive 3-factorizations in $S_{mn}$ with cycle lengths in $\sigma_1$ all divisible by $m$.

For $n \geq 1$ and $\lambda$ a partition of $n$, we note $C_\lambda$ the set of permutations with cycle type $\lambda$. For a permutation $\phi$, we note $l(\phi)$ the number of cycles in $\phi$. For a partition $\mu = [1^{m_1}2^{m_2}\ldots]$, we note $y_\mu=\prod_{i>0} y_i^{m_i}$. We can now define the following generating series $R_H$ of (not necessarily transitive) $3$-factorizations $\sigma \tau \phi = id$ in $S_{mn}$ with cycle lengths in $\phi$ all divisible by $m$ and $\tau = (1,2,\ldots,m)(m+1,\ldots,2m)\ldots((n-1)m+1, \ldots, mn)$ a fixed partition of cycle type $[m^n]$, with $x$ marking the number of cycles in $\sigma$, $y$ the cycle type of $\phi$ and $z$ the size of the symmetric group:
\begin{displaymath}
R_H(x,\underline{\mathbf{y}},z) = \sum_{n \geq 1} \frac{z^n}{n!} \sum_{\mu \vdash n} \sum_{\substack{\sigma \tau \phi = id_{mn} \\ \phi \in C_{m\mu}}} x^{l(\sigma)} y_{\mu}.
\end{displaymath}
Similarly, we can also define the generating series $R_C$ of (not necessarily transitive) $(m+1)$-factorizations $\sigma_1 \cdots \sigma_m \phi = id$ in $S_n$, with $x_i$ marking the number of cycles in $\sigma_i$, $y$ the cycle type of $\phi$ and $z$ the size of the symmetric group:  
\begin{displaymath}
R_C(\underline{\mathbf{x}},\underline{\mathbf{y}},z) = \sum_{n \geq 1} \frac{z^n}{n!} \sum_{\mu \vdash n} \sum_{\substack{\sigma_1 \ldots \sigma_m \phi = id_n \\ \phi \in C_{\mu}}} y_{\mu} \prod_{i=1}^{m} x_i^{l(\sigma_i)}.
\end{displaymath}

By taking the logarithm of the corresponding generating series, we can pass from general $k$-factorizations to transitive ones (see, \textit{e.g.}, \cite{jackson1990character}). Using Euler's formula, we can now easily verify the following relations concerning generating series $H,C$ of $m$-hypermaps and $m$-constellations, and $R_H, R_C$ defined above:

\begin{equation} \label{eq:H-to-RH}
H(x,\underline{\mathbf{y}},z,u) = mu^2 \left( z \frac{\partial}{\partial z} (\log R_H) \right) (xu^{-1}, [y_i \gets y_i u^{-1}], \frac{1}{m}zu^{m-1}),
\end{equation}
\begin{equation} \label{eq:C-to-RC}
C(\underline{\mathbf{x}},\underline{\mathbf{y}},z,u) = u^2 \left( z \frac{\partial}{\partial z} (\log R_C) \right) ([x_i \gets x_i u^{-1}], [y_i \gets y_i u^{-1}], zu^{m-1}).
\end{equation}

In an algebraic point of view, the series $R_H$ and $R_C$ are much easier to manipulate than $H$ and $C$. To investigate the link between $m$-hypermaps and $m$-constellations, we will start by analyzing $R_H$ and $R_C$ using the group algebra of the symmetric group.

\subsection{Characters and group algebra of the symmetric group}

The group algebra $\complexes S_n$ of the symmetric group $S_n$ is a complex vector space with a canonical basis indexed by elements of $S_n$ and a multiplication of elements extending distributively the group law of $S_n$. The center of the group algebra $\complexes S_n$, noted as $Z(\complexes S_n)$, is the subalgebra of $\complexes S_n$ consisting of elements that commute with any element in $\complexes S_n$. 

For $\theta$ a partition of $n$ (noted as $\theta \vdash n$), we define $K_\theta$ to be the formal sum of elements in $S_n$ with cycle type $\theta$. The elements $(K_\theta)_{\theta \vdash n}$ form a linear basis of $Z( \complexes S_n)$. According to the classic representation theory (\textit{c.f.} Chapter~2.5 in \cite{serre1977linear}), $Z( \complexes S_n)$ has another linear basis $(F_\theta)_{\theta \vdash n}$ formed by orthogonal idempotents.

For a partition $\lambda = [1^{m_1} 2^{m_2} \ldots] \vdash n$ in which $i$ appears $m_i$ times, we note $z_\lambda = \prod_{i > 0} i^{m_i} m_{i}!$, and we know that $n! z_{\lambda}^{-1}$ is the number of permutations of cycle type $\lambda$. We denote by $\chi^{\lambda}_{\theta}$ the irreducible character indexed by $\lambda$ evaluated on the conjugacy class of cycle type $\theta$, and by $f^\lambda = \chi^{\lambda}_{[1^n]}$ the dimension of the irreducible representation indexed by $\lambda$ (\textit{c.f.} \cite{stanley2001enumerative}). The change of basis between $(K_\theta)_{\theta \vdash n}$ and $(F_\theta)_{\theta \vdash n}$ is thus given by (\textit{c.f.} Chapter~2.5 in \cite{serre1977linear})
\[ F_\lambda = f^{\lambda} (n!)^{-1} \sum_{\theta \vdash n} \chi^{\lambda}_{\theta}K_\theta, \quad K_\lambda = n! z_{\lambda}^{-1} \sum_{\theta \vdash n} \chi^{\theta}_{\lambda} (f^{\theta})^{-1} F_\theta. \]

We now work in $ Z(\complexes S_n) $. For arbitrary partitions $\alpha, \beta^{(1)}, \ldots, \beta^{(k)}$ of $n$, we consider the coefficient of $K_\alpha$ in $K_{\beta^{(1)}} \cdots K_{\beta^{(k)}}$ expressed as a vector in $ Z(\complexes S_n) $ under the basis $(K_\theta)_{\theta \vdash n}$. This coefficient, noted as $[K_\alpha]K_{\beta^{(1)}} \cdots K_{\beta^{(k)}}$, can be interpreted as the number of factorizations $\tau_1 \cdots \tau_k \sigma = id$ with $\tau_i$ of cycle type $\beta^{(i)}$ for each $1 \leq i \leq k$ and $\sigma$ a fixed permutation of cycle type $\alpha$. With this interpretation, using the change of basis between $(K_\theta)_{\theta \vdash n}$ and $(F_\theta)_{\theta \vdash n}$ given above and the fact that $(F_\theta)_{\theta \vdash n}$ are orthogonal idempotents, we have

\[
\sum_{\substack{\tau_1 \cdots \tau_k \sigma = id \\ \sigma \in C_\alpha \mathrm{fixed}, \forall i, \tau_i \in C_{\beta^{(i)}}}} 1  = \sum_{\theta \vdash n} \left( n! \right)^{-1} \left( f^{\theta} \right)^{1-k} \chi_{\alpha}^{\theta} \prod_{i=1}^{k} \left( n! z_{\beta^{(i)}} \chi_{\beta^{(i)}}^{\theta} \right).
\]

We have just reproved a specialization of a well-known formula that is often contributed to Frobenius (\textit{c.f.} Appendix~A in \cite{lando2004graphs}), with which we can rewrite $R_H$ and $R_C$ as following.

\begin{equation} \label{eq:RH-in-all-characters}
R_H(x,\underline{\mathbf{y}},z) = \sum_{n \geq 1} \frac{z^n}{n!} \sum_{\lambda \vdash mn, \mu \vdash n} n! z_{\lambda}^{-1} z_{m\mu}^{-1} x^{l(\lambda)} y_{\mu} \sum_{\theta \vdash mn} \frac{1}{f^{\theta}} \chi_{\lambda}^{\theta} \chi_{[m^n]}^{\theta} \chi_{m\mu}^{\theta}
\end{equation}
\begin{equation} \label{eq:RC-in-all-characters}
R_C(\underline{\mathbf{x}},\underline{\mathbf{y}},z) = \sum_{n \geq 1} \frac{z^n}{n!} \sum_{\lambda^{(1)}, \ldots, \lambda^{(m)}, \mu \vdash n} \left( \prod_{i=1}^{m} x_i^{l(\lambda_{(i)})} \right) y_{\mu} \sum_{\theta \vdash n} (f^{\theta})^{(1-m)} z_{\mu}^{-1} \chi^{\theta}_{\mu} \prod_{i=1}^{k} n! z_{\lambda^{(i)}}^{-1} \chi^{\theta}_{\lambda^{(i)}}
\end{equation}

To further simplify the expressions above, we define the \emph{rising factorial function} $x^{(n)} = x(x+1) \cdots (x+n-1)$ for $n \in \naturals$. For a partition $\theta$, we define the polynomial $H_\theta(x)$ as $\prod_{i=1}^{l(\theta)} (x-i+1)^{(\theta_i)}$, where $l(\theta)$ is the number of parts in $\theta$. With this notation, we give the following expressions of $R_H$ and $R_C$.

\begin{prop} \label{prop:series-in-big-characters-simplified}
We can rewrite $R_H$ and $R_C$ as follows:
\[ R_H(x,\underline{\mathbf{y}},z) = \sum_{n \geq 1} \frac{z^n}{n!} \sum_{\mu \vdash n} m^{-l(\mu)} z_{\mu}^{-1} y_{\mu} \sum_{\theta \vdash mn} \chi_{[m^n]}^{\theta} \chi_{m\mu}^{\theta} H_{\theta}(x), \]
\[ R_C(\underline{\mathbf{x}},\underline{\mathbf{y}},z) = \sum_{n \geq 1} \frac{z^n}{n!} \sum_{\mu \vdash n} y_{\mu}  z_{\mu}^{-1} \sum_{\theta \vdash n} f^{\theta}  \chi^{\theta}_{\mu} \left( \prod_{i=1}^{m} H_\theta(x_i) \right). \]
\end{prop}

This proposition comes from direct application of the following lemma (Lemma 3.4 in \cite{jackson1990character}) to \eqref{eq:RH-in-all-characters} and \eqref{eq:RC-in-all-characters}.

\begin{lem} \label{lem:polynomial-H}
We have the following equality:
\[ n! \sum_{\alpha \vdash n} z_{\alpha}^{-1} \chi_{\alpha}^{\theta} x^{l(\alpha)} = f^{\theta}H_{\theta}(x). \]
\end{lem}

Since no proof is directly given in \cite{jackson1990character}, for self-containedness, we present here our proof using the representation theory of the symmetric group as in \cite{vershik2004new}. In the following proof, we identify a partition $\theta$ with its Ferrers diagram, and for a cell $w \in \theta$ in row $i$ and column $j$, we note $c(w) = j - i$ its \emph{content}.

\begin{proof}
For $1 \leq i \leq n$, we note $J_i = \sum_{j < i}(j \; i)$ the $i$-th \emph{Jucys-Murphy element}. Consider $T(x) = \prod_{i=1}^{n} (x + J_i)$, we can see that $T(x) = \sum_{\sigma \in S_n} x^{l(\sigma)} \sigma = \sum_{\alpha \vdash n} x^{l(\alpha)} K_\alpha$. By the change of basis from $(K_\theta)_{\theta \vdash n}$ to $(F_\theta)_{\theta \vdash n}$ and the fact that $(F_\theta)_{\theta \vdash n}$ are orthogonal idempotents, we have

\begin{displaymath}
T(x)F_\theta = \sum_{\alpha \vdash n} x^{l(\alpha)} h^{\alpha} \frac{\chi_{\alpha}^{\theta}}{f^{\theta}} F_{\theta}. 
\end{displaymath}

On the other hand, $F_\theta$ is the projection of $\complexes S_n$ on the irreducible module $M_\theta$ indexed by the partition $\theta$. For a vector $v_T$ in the Gelfand-Tsetlin basis of $M_\theta$, we have $T(x)v_T = \prod_{w \in \theta}(x+c(w)) v_T$, since all $J_i$ acts diagonally on the Gelfand-Tsetlin basis with $J_i v_T = c(i) v_T$ where $c(i)$ is the content of the cell occupied by $i$ in the Young tableau $T$ (\textit{c.f.} \cite{vershik2004new}). Thus $T(x)$ acts on $M_\theta$ as $\prod_{w \in \theta}(x+c(w))$. Therefore, the operators $T(x)F_\theta$ and $\prod_{w \in \theta}(x+c(w)) F_\theta$ act identically in $\complexes S_n$ and thus they are equal. We conclude the proof by comparing the two equalities with the observation that $\prod_{w \in \theta}(x+c(w)) = H_\theta(x)$.
\end{proof}

We can see that the characters in $R_H$ in Proposition \ref{prop:series-in-big-characters-simplified} are all evaluated at the partition $[m^n]$ and partitions of the form $m\mu$ with $\mu \vdash n$. In \cite{jackson1990character}, $\chi_{[m^n]}^{\theta}$ is proved to have an expression as a product of smaller characters, which is a crucial step towards the quadrangulation relation. This factorization is also presented in \cite{james1981representation} (Section 2.7) under the framework of $p$-core and abacus display of a partition. With a generalization due to Littlewood \cite{littlewood1951modular} that applies to all partitions of the form $m\mu$, we will give a similar relation between $m$-hypermaps and $m$-constellations in Section~\ref{sec:app}.

\section{Factorization of characters evaluated at $m\lambda$} \label{sec:fact}

In this section we will present the following result on factorizing $\chi_{[m\lambda]}^{\theta}$ into smaller characters. The notion of $m$-splittable partition will be defined later.

\begin{thm}[Littlewood 1951 \cite{littlewood1951modular}] \label{thm:character-factorization}
Let $m,n$ be two natural numbers, and $\lambda \vdash n$, $\theta \vdash mn$ be two partitions. We consider partitions as multisets and we denote multiset sum by $\uplus$. If $\theta$ is $m$-splittable, we have
\[ \chi_{m\lambda}^{\theta} = z_{\lambda} \sgn_\theta \sum_{\lambda^{(1)} \uplus \cdots \uplus \lambda^{(m)} = \lambda} \prod_{i=1}^m \chi_{\lambda^{(i)}}^{\theta^{(i)}} z_{\lambda^{(i)}}^{-1}, \]
with $\sgn_\theta$ and all $\theta^{(i)}$ depending only on $\theta$ and $m$.

If $\theta$ is not $m$-splittable, $\chi_{m\lambda}^{\theta} = 0$.
\end{thm}

An algebraic proof is given in \cite{littlewood1951modular}. For the sake of self-containedness, we will present a combinatorial proof here. We will first give a natural combinatorial interpretation of $m$-splittable partitions using the infinite wedge space. A brief introduction to the infinite wedge space can be found in the appendix of \cite{okounkov2001infinite}, after which some of our notations here follow. With this combinatorial interpretation, we will give a straightforward, purely combinatorial proof of Theorem \ref{thm:character-factorization}.

\subsection{Infinite wedge space and boson-fermion correspondence}
We recall some definitions about the infinite wedge space taken from \cite{okounkov2001infinite}. Let $(\underline{k})_{k \in \integers}$ be a set of variables indexed by integers, and $\wedge$ be an associative and anti-commutative binary relation acting as exterior product. For $S \subset \integers$, we note $S_+ = S \cap \naturals$ and $S_- = \integers_{< 0} \setminus S$. We define $\Lambda^{\infty/2}$ as the vector space spanned by vectors of the form $v_{S} = \underline{s_1} \wedge \underline{s_2} \wedge \ldots$ with $S = \{ s_1 > s_2 > \ldots \}$ such that both $S_+$ and $S_-$ are finite. The vector space $\Lambda^{\infty/2}$ is called the \emph{infinite wedge space}.

We define the creation operator $\phi_k$ corresponding to the variable $\underline{k}$ with $\phi_k(v) = \underline{k} \wedge v$ for all $v \in \Lambda^{\infty/2}$. We also define the annihilation operator $\phi^{*}_k$ as the adjoint operator of $\phi_k$ with respect to the canonical scalar product, satisfying the anti-commutation relation $ \phi^{*}_k \phi_k + \phi_k \phi^{*}_k = 1 $. Let $\Lambda_0$ be the subspace spanned by vectors of the form $v_{S}$ satisfying that $S_+$ and $S_-$ are finite and $|S_+|=|S_-|$. This condition on $S$ is called the \emph{charge condition} hereinafter.

\newcommand{\msq}[1]{rectangle +(0.5,0.5) +(0.25,0.25) node{#1}}
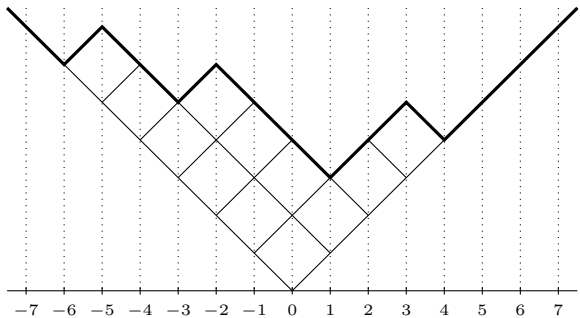
\begin{figure}[!htbp]
\begin{center}
\begin{tikzpicture}[scale = 0.5]
\draw (-7.5, 7.5) -- (0,0) -- (7.5, 7.5);
\draw (-7.5, 0) -- (7.5, 0);
\foreach \x in {-7, -6, ..., 7} \draw[style=dotted] (\x, 7.5) -- (\x, 0);
\foreach \x in {-7, -6, ..., 7} \draw (\x, 0.1) -- (\x, -0.1) node[anchor=north] {\tiny $\x$};
\draw[style = very thick] (-7.5, 7.5) -- (-6, 6) -- (-5, 7) -- (-3, 5) -- (-2, 6) -- (1, 3) -- (3, 5) -- (4, 4) -- (7.5, 7.5);
\draw (-5, 5) -- (-4, 6);
\draw (-4, 4) -- (-3, 5);
\draw (-3, 3) -- (-1, 5);
\draw (-2, 2) -- (0, 4);
\draw (-1, 1) -- (1, 3);
\draw (1, 1) -- (-3, 5);
\draw (2, 2) -- (1, 3);
\draw (3, 3) -- (2, 4);
\end{tikzpicture}
\end{center}
\caption{Rotated diagram of the partition $\lambda=(4,2,2,2,1,1)$ corresponding to the vector $v_S$ with $S=\{ 3, 0, -1, -2, -4, -5, -7, -8 \ldots \}$} \label{fig:russian-notation}
\end{figure}

Partitions are in bijection with sets satisfying the charge condition. Given a partition $\lambda = (\lambda_1, \lambda_2, \ldots, \lambda_l)$, we draw its diagram in French convention, then rotate the diagram by 45 degrees. The diagram is framed by a lattice path (thick line in Figure \ref{fig:russian-notation}), called \emph{the framing path}, consisting of two types of steps, one parallel to the line $y=x$ and the other to $y=-x$. This framing path eventually coincides with the line $y = -x$ to the left and $y=x$ to the right. We can also consider $\lambda$ as an infinite sequence by taking $\lambda_k = 0$ for $k>l$. We define the set $S_\lambda = \{ \lambda_i - i \mid i \in \naturals \} $. We can see from Figure \ref{fig:russian-notation} that $S_\lambda$ is exactly the set of starting abscissas of down-going steps (parallel to $y=-x$).

This map from $\lambda$ to $S_\lambda$ is a classical bijection between partitions and sets satisfying the charge condition. We illustrate the set $S_\lambda$ on figure as a diagram of $\mathbb{Z}$ where each position indexed by an element of $S_\lambda$ is occupied by a particle. This bijection is called the \emph{boson-fermion correspondence} in the literature. More information about representing partitions by its framing path can be found in \cite{okounkov2001infinite}.

For a partition $\lambda$, we note $v_\lambda$ the vector $v_{S_\lambda}$ corresponding to $S_\lambda$. We now define a new operator $\sigma_{n,k} = \phi_{k} \phi^{*}_{n+k}$ for $n$ a positive integer and $k$ an integer. The effect of $\sigma_{n,k}$ on $v_\lambda$ is exactly removing a border strip of length $n$ from the appropriate position of $\lambda$ when possible, and we have $\sigma_{n,k} v_\lambda = (-1)^{ht(\lambda / \mu)} v_\mu$ with $\mu$ the partition after removal of the border strip and $ht(\lambda / \mu)$ the number of rows the border strip spans minus one. See Figure \ref{fig:border-strip-removal} for an example.

Combinatorially, the operator $\sigma_{k,p}$ can be viewed as a jump of a particle from position $n+k$ back to position $k$, and the sign it induces corresponds to the number of particles underneath this jump, that is, the number of jump-overs of this jump. For example, in Figure \ref{fig:border-strip-removal}, three jump-overs occur. We can also see that $ht(\lambda / \mu)$ is exactly the number of jump-overs in the removal of the border strip $\lambda / \mu$.

\begin{figure}[!htbp]
\begin{center}
\begin{tikzpicture}[scale = 0.5]
\draw (-7.5, 7.5) -- (0,0) -- (7.5, 7.5);
\draw (-7.5, 0) -- (7.5, 0);
\foreach \x in {-7, -6, ..., 7} \draw[style=dotted] (\x, 7.5) -- (\x, 0);
\foreach \x in {-7, -6, ..., 7} \draw (\x, 0.1) -- (\x, -0.1) node[anchor=north] {\tiny $\x$};
\draw[style = dashed] (-7.5, 7.5) -- (-6, 6) -- (-5, 7) -- (-3, 5) -- (-2, 6) -- (1, 3) -- (3, 5) -- (4, 4) -- (7.5, 7.5);
\draw[style = very thick] (-7.5, 7.5) -- (-3, 3) -- (-2, 4) -- (-1, 3) -- (0, 4) -- (1, 3) -- (3, 5) -- (4, 4) -- (7.5, 7.5);
\draw[dashed] (-5, 5) -- (-4, 6);
\draw[dashed] (-4, 4) -- (-3, 5);
\draw[dashed] (-3, 3) -- (-1, 5);
\draw (-2, 2) -- (0, 4);
\draw (-1, 1) -- (1, 3);
\draw (1, 1) -- (-2, 4);
\draw[dashed] (-2, 4) -- (-3, 5);
\draw (2, 2) -- (1, 3);
\draw (3, 3) -- (2, 4);

\draw (-7.5, -2.5) -- (7.5, -2.5);
\foreach \x in {3, 0, -6, -2, -4, -5, -7} \filldraw[black] (\x, -2.5) circle (0.2);
\draw[dashed] (-1, -2.5) circle (0.25);
\draw[-latex] (-1, -2.1) -- (-1, -1) -- (-6, -1) -- (-6, -2.1);
\draw[dotted] (0, -2) -- (0, -3);
\end{tikzpicture}
\end{center}
\caption{Effect of $\sigma_{-6,5}$ on $v_\lambda$ with $\lambda = (4,2,2,2,1,1)$. We have $\sigma_{-6,5} v_\lambda = - v_\mu$ with $\mu = (4,2,1)$} \label{fig:border-strip-removal}
\end{figure}
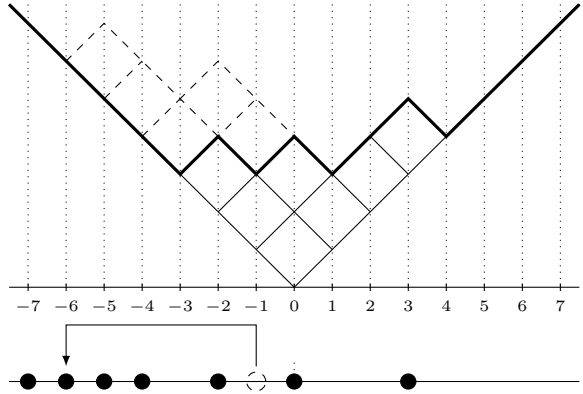

We now define a new operator $\alpha_n = \sum_{k \in \integers} \sigma_{n,k}$ for $n \neq 0$, which essentially tries to remove a border strip of length $n$ from $\lambda$ in all possible ways when applied to $v_\lambda$. It follows from the Murnaghan-Nakayama rule (\textit{c.f.} \cite{stanley2001enumerative}) that $\chi^{\theta}_\lambda$ is the coefficient of $v_{(0)}$, which corresponds to the empty partition, in $\alpha_{\lambda_1} \cdots \alpha_{\lambda_{l(\lambda)}} v_\theta$.

\subsection{$m$-splittable partitions}
We now define $m$-splittable partitions. Let $S$ be a set satisfying the charge condition. We define its \emph{$m$-split} as an $m$-tuple of sets $(S_0, S_1, \ldots, S_{m-1})$ such that $S_i = \{ a \, | \, ma+i \in S \}$ for $i$ from $0$ to $m-1$. A set $S$ satisfying the charge condition is called \emph{$m$-splittable} if every set in its $m$-split satisfies the charge condition. A partition $\lambda$ is called  \emph{$m$-splittable} if $S_\lambda$ is $m$-splittable. In this case, we define the \emph{$m$-split} $(\lambda^{(0)}, \ldots, \lambda^{(m-1)})$ of $\lambda$ to be the tuple of partitions corresponding component-wise to the $m$-split of $S_\lambda$.

Here is an example in Figure \ref{fig:m-split} of the $m$-split of an $m$-splittable partition. We take $m=3$ and we consider the partition $\theta = (6,6,4,4,4,3,3)$. We can verify easily that $\theta$ is 3-splittable. To obtain the 3-split of $\theta$, we split the set $S_\theta$ according to residue classes modulo 3, then rescale to obtain 3 smaller sets, and finally we reconstruct partitions corresponding to the smaller sets.

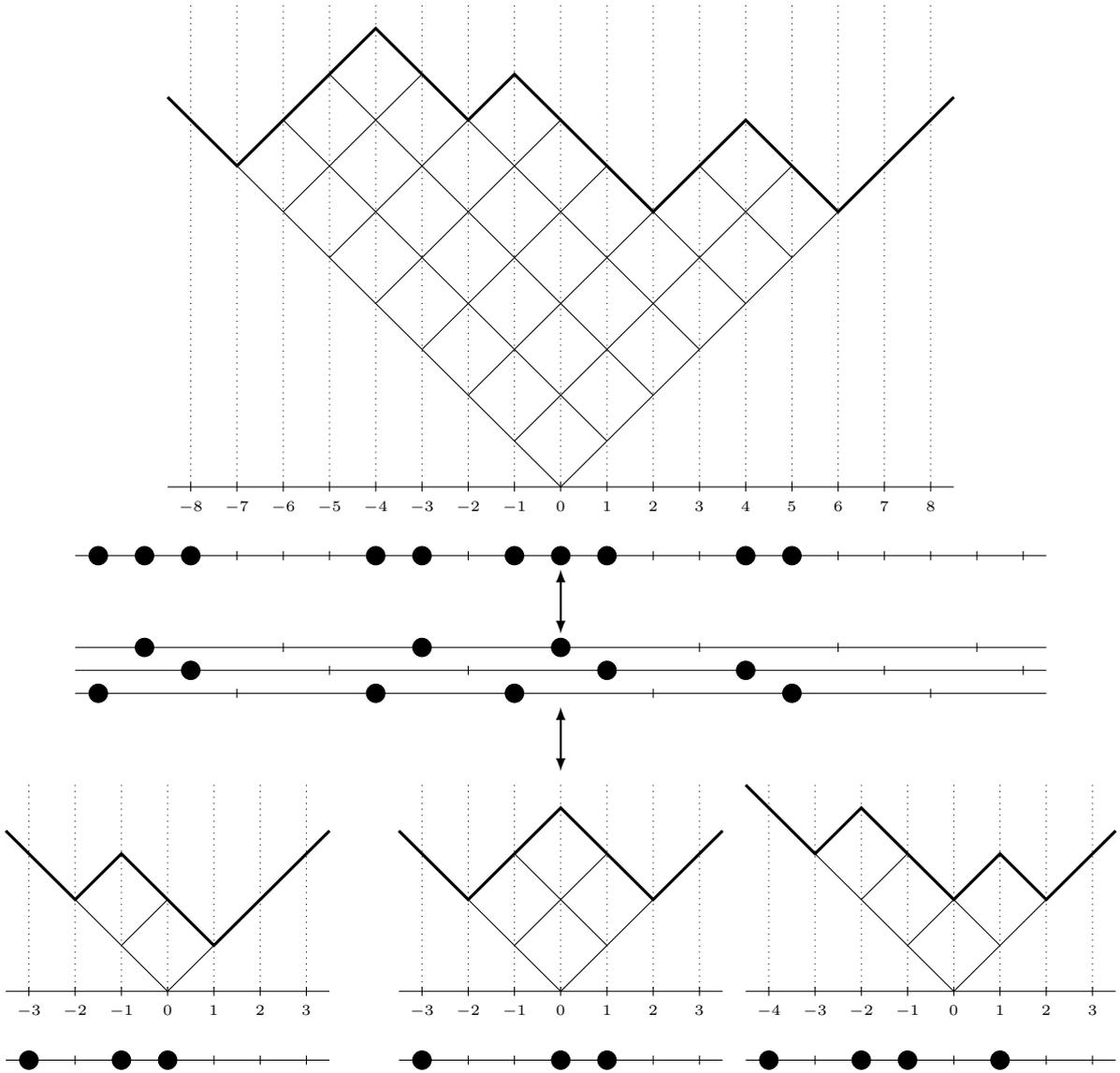
\begin{figure}[!htbp]
\begin{center}
\begin{tikzpicture}[scale = 0.65]
\draw (-8.5, 8.5) -- (0,0) -- (8.5, 8.5);
\draw (-8.5, 0) -- (8.5, 0);
\foreach \x in {-8, -7, ..., 8} \draw[style=dotted] (\x, 10.5) -- (\x, 0);
\foreach \x in {-8, -7, ..., 8} \draw (\x, 0.1) -- (\x, -0.1) node[anchor=north] {\tiny $\x$};
\draw[very thick] (-8.5, 8.5) -- (-7, 7) -- (-4, 10) -- (-2, 8) -- (-1, 9) -- (2, 6) -- (4, 8) -- (6, 6) -- (8.5, 8.5);
\draw (-6, 6) -- (-3, 9);
\draw (-5, 5) -- (-2, 8);
\draw (-4, 4) -- (0, 8);
\draw (-3, 3) -- (1, 7);
\draw (-2, 2) -- (2, 6);
\draw (-1, 1) -- (5, 7);
\draw (1, 1) -- (-6, 8);
\draw (2, 2) -- (-5, 9);
\draw (3, 3) -- (-2, 8);
\draw (4, 4) -- (2, 6);
\draw (5, 5) -- (3, 7);

\draw (-10.5, -1.5) -- (10.5, -1.5);
\foreach \x in {-10, -9, ..., 10} \draw (\x, -1.6) -- (\x, -1.4);
\foreach \x in {5, 4, 1, 0, -1, -3, -4, -8, -9, -10} \filldraw[black] (\x, -1.5) circle (0.2);

\draw[latex-latex, thick] (0, -1.8) -- (0, -3.2);

\draw (-10.5, -3.5) -- (10.5, -3.5);
\foreach \x in {-9, -6, ..., 9} \draw (\x, -3.4) -- (\x, -3.6);
\foreach \x in {0, -3, -9} \filldraw[black] (\x, -3.5) circle (0.2);

\draw (-10.5, -4) -- (10.5, -4);
\foreach \x in {-8, -5, ..., 10} \draw (\x, -3.9) -- (\x, -4.1);
\foreach \x in {4, 1, -8} \filldraw[black] (\x, -4) circle (0.2);

\draw (-10.5, -4.5) -- (10.5, -4.5);
\foreach \x in {-10, -7, ..., 8} \draw (\x, -4.4) -- (\x, -4.6);
\foreach \x in {5, -1, -4, -10} \filldraw[black] (\x, -4.5) circle (0.2);

\draw[latex-latex, thick] (0, -4.8) -- (0, -6.2);

\begin{scope}[xshift=-8.5cm, yshift=-11cm]
\draw (-3.5, 3.5) -- (0,0) -- (3.5, 3.5);
\draw (-3.5, 0) -- (3.5, 0);
\foreach \x in {-3, -2, ..., 3} \draw[style=dotted] (\x, 4.5) -- (\x, 0);
\foreach \x in {-3, -2, ..., 3} \draw (\x, 0.1) -- (\x, -0.1) node[anchor=north] {\tiny $\x$};
\draw[very thick] (-3.5, 3.5) -- (-2,2) -- (-1,3) -- (1,1) -- (3.5,3.5);
\draw (-1,1) -- (0, 2);
\draw (-3.5, -1.5) -- (3.5, -1.5);
\foreach \x in {-3, -2, ..., 3} \draw (\x, -1.6) -- (\x, -1.4);
\foreach \x in {-3, -1, 0} \filldraw[black] (\x, -1.5) circle (0.2);
\end{scope}

\begin{scope}[yshift=-11cm]
\draw (-3.5, 3.5) -- (0,0) -- (3.5, 3.5);
\draw (-3.5, 0) -- (3.5, 0);
\foreach \x in {-3, -2, ..., 3} \draw[style=dotted] (\x, 4.5) -- (\x, 0);
\foreach \x in {-3, -2, ..., 3} \draw (\x, 0.1) -- (\x, -0.1) node[anchor=north] {\tiny $\x$};
\draw[very thick] (-3.5, 3.5) -- (-2,2) -- (0,4) -- (2,2) -- (3.5,3.5);
\draw (-1,1) -- (1, 3);
\draw (1,1) -- (-1, 3);
\draw (-3.5, -1.5) -- (3.5, -1.5);
\foreach \x in {-3, -2, ..., 3} \draw (\x, -1.6) -- (\x, -1.4);
\foreach \x in {-3, 0, 1} \filldraw[black] (\x, -1.5) circle (0.2);
\end{scope}

\begin{scope}[xshift=8.5cm, yshift=-11cm]
\draw (-4.5, 4.5) -- (0,0) -- (3.5, 3.5);
\draw (-4.5, 0) -- (3.5, 0);
\foreach \x in {-4, -3, ..., 3} \draw[style=dotted] (\x, 4.5) -- (\x, 0);
\foreach \x in {-4, -3, ..., 3} \draw (\x, 0.1) -- (\x, -0.1) node[anchor=north] {\tiny $\x$};
\draw[very thick] (-4.5, 4.5) -- (-3,3) -- (-2,4) -- (0,2) -- (1,3) -- (2,2) -- (3.5,3.5);
\draw (-2, 2) -- (-1,3);
\draw (-1,1) -- (0, 2);
\draw (1,1) -- (0, 2);
\draw (-4.5, -1.5) -- (3.5, -1.5);
\foreach \x in {-4, -3, ..., 3} \draw (\x, -1.6) -- (\x, -1.4);
\foreach \x in {-4, -2, -1, 1} \filldraw[black] (\x, -1.5) circle (0.2);
\end{scope}
\end{tikzpicture}
\end{center}
\caption{Example of a 3-splittable partition, alongside with its 3-split} \label{fig:m-split}
\end{figure}

As a remark, comparing our terminology with the one in \cite{littlewood1951modular}, it is easy to see that an $m$-splittable partition is exactly a partition with an empty $m$-core, and in this case, its $m$-split coincides with its $m$-quotient. Moreover, we can easily show that the notion of ``$m$-balanced partitions'' used in \cite{jackson1990character} and \cite{jackson1999combinatorial} is exactly the notion of $m$-splittable partitions. An advantage of our point of view is that it is much more intuitive and avoids technical lemmas when dealing with these objects as in \cite{jackson1990character}.

\subsection{Combinatorial proof of Theorem \ref{thm:character-factorization}}
We are now ready to give a combinatorial proof to Theorem \ref{thm:character-factorization}, alongside with explicit expression of $\sgn_\theta$ and all $\theta^{(i)}$. Essentially, using the Murnaghan-Nakayama rule, we establish a bijection between ribbon tableaux of shape $\theta$ and content $m\lambda$ and sequences of $m$ ribbon tableaux $T_0, \ldots, T_{m-1}$ of shape $\theta^{(0)}, \ldots, \theta^{(m-1)}$ respectively and total content $\lambda$. Readers can refer to \cite{stanley2001enumerative} for more on ribbon tableaux and the Murnaghan-Nakayama rule.

The basic idea is that the removal of a border strip $s$ of length multiple of $m$ from the partition $\lambda$ only affects elements in $S_\lambda$ in one congruence class modulo $m$. Since one congruence class corresponds to one component in the $m$-split, the removal of $s$ can be reduced to the removal of a smaller strip $s'$ in the corresponding component in the $m$-split. We then deal with the sign issue.

\begin{proof}[Combinatorial proof of Theorem \ref{thm:character-factorization}]

We try to evaluate $\chi^\theta_{m\lambda}$ with the Murnaghan-Nakayama rule.

For any integer $p, k$ with $k > 0$, the operator $\sigma_{mk, p}$ only affects particles occupying the $n$-th position with $n \equiv p \mod m$. For any $S$, $\sigma_{mk, mp+i}$ only changes the component $S_{i}$ in the $m$-split, and its effect is equivalent to $\sigma_{k, p}$ for $S_i$. Thus the operator $\alpha_{mk}$ preserves the $m$-splittable property of a partition. We then have $\chi^\theta_{m\lambda} = 0$ for $\theta$ not $m$-splittable, since the empty partition is $m$-splittable.

Now we suppose that $\theta$ is $m$-splittable. Let $( \theta^{(0)}, \ldots, \theta^{(m-1)} )$ be its $m$-split.

A ribbon tableau $T_\theta$ of shape $\theta$ and of content $[1^{m\lambda_1}2^{m\lambda_2}\ldots]$ can be considered as the application of a sequence of operators of the form $\sigma_{mk, p}$ on $v_\theta$, therefore $m$ series of operators of the form $\sigma_{k,p}$ on all of the $\theta^{(i)}$. The operator $\sigma_{mk, mp+i}$ on $v_\theta$ acts only on $\theta^{(i)}$, and acts as the operator $\sigma_{k,p}$. This induces a bijection that sends a ribbon tableau of shape $\theta$ and content $m\lambda$ to $m$ ribbon tableaux $T_0, \ldots, T_{m-1}$ of shapes $\theta^{(0)}, \ldots, \theta^{(m-1)}$ and total content $\lambda$. We note $\lambda^{(0)}, \ldots, \lambda^{(m-1)}$ the content of $T_0, \ldots, T_{m-1}$ respectively. 

We now consider the sign. For a ribbon tableau $T$, we note $\sgn(T)$ the sign of a ribbon tableau $T$, which is the product of $(-1)^{ht(s)}$ for all strips $s$. We also note $j(T)$ the number of jump-overs in the corresponding operator sequence, and we have $\sgn(T) = (-1)^{j(T)}$. In $T_\theta$, there are two types of jump-overs: jump-overs between particles in the same congruence class, and jump-overs between particles in different congruence classes. The number of jump-overs of the first type is noted as $j_{\mathrm{endo}}(T)$, and that of the second type $j_{\mathrm{inter}}(T)$, and we have $j(T) = j_{\mathrm{endo}}(T) + j_{\mathrm{inter}}(T)$. By definition $j_{\mathrm{endo}}(T) = \sum_{i=0}^{m-1} j(T_i)$. For $j_{\mathrm{inter}}(T)$, we can check that the parity of $j_{\mathrm{inter}}(T)$ is preserved when commuting any two operators $\sigma_{mk,p_1}$ and $\sigma_{mk,p_2}$, and when replacing any operator $\sigma_{m(k+l), p}$ by $\sigma_{ml, p-mk} \sigma_{mk,p}$. Therefore, the parity of $j_{\mathrm{inter}}(T)$ depends only on $\theta$. We thus define $\sgn_\theta$ as $\sgn_\theta = (-1)^{j_{\mathrm{inter}}(T)}$ and we have $\sgn(T) =\sgn_\theta \prod_{i=0}^{m-1} \sgn(T_i) $.

To evaluate $\chi^\theta_{m\lambda}$ with the Murnaghan-Nakayama rule, we consider the sum over the sign of all ribbon tableau $T_\theta$ of shape $\theta$ and of content $m\lambda$. We note $t_k$ the multiplicity of $k$ as parts in $\lambda$, and $t_{k,i}$ the multiplicity of $k$ in $\lambda^{(i)}$. We have $t_k = \sum_{j=0}^{m-1}t_{k,j}$. By the bijection and the sign relation between $T$ and $T_0, \ldots, T_{m-1}$ mentioned above, we have the following formula:

\begin{displaymath}
\sum_{T_\theta} \sgn(T_\theta) = \sum_{\lambda^{(1)} \uplus \cdots \uplus \lambda^{(m)} = \lambda} \sum_{T_0, \ldots, T_{m-1}} \sgn_\theta \left( \prod_{i=0}^{m-1} \sgn(T_i) \right) \prod_{k > 0}  \binom{t_k}{t_{k,0}, \ldots, t_{k,m-1}}.
\end{displaymath}

This is essentially a double-counting formula. On the left hand side, we have the sum of the sign of all ribbon tableaux $T_\theta$ of shape $\theta$ and of content $m\lambda$. On the right hand side, we first sum over all possible ways of regrouping parts of $\lambda$ into smaller partitions $\lambda^{(0)}, \ldots, \lambda^{(m-1)}$, then we sum over all tuples of ribbon tableaux $T_0, T_1, \ldots, T_{m-1}$ with $T_i$ of content $\lambda^{(i)}$, and in the final summand, we have the product of $\sgn_\theta$ and the signs of all $T_i$. We have the multinomial factor in the final summand since there are multiple ways to distribute operators $\alpha_{n}$ with the same $n$ to achieve the same sequence of $\lambda^{(0)}, \ldots, \lambda^{(m-1)}$. Since we have a bijection sending $T_\theta$ to $T_0, \ldots, T_{m-1}$, and a relation $\sgn(T) =\sgn_\theta \prod_{i=0}^{m-1} \sgn(T_i) $ on their signs, we have the equality.

We conclude the proof by the Murnaghan-Nakayama rule while noticing that $z_\lambda=\prod_{k>0}k^{t_k}t_{k}!$, and the fact that
\begin{displaymath}
z_\lambda \prod_{i=0}^{m-1} z^{-1}_{\lambda^{(i)}} = \prod_{k>0} k^{t_k - \sum_{j=0}^{m-1}t_{k,j}} \frac{t_{k}!}{t_{k,0}! t_{k,1}! \cdots t_{k,m-1}!} = \prod_{k>0} \binom{t_k}{t_{k,0}, \ldots, t_{k,m-1}}.
\end{displaymath}
\end{proof}

As an application of this combinatorial point of view, we also have a factorization result on the polynomial $H_\theta$ for $m$-splittable partitions $\theta$.

\begin{lem} \label{lem:from-Hm-back-to-H}
For an $m$-splittable partition $\theta \vdash n$, we have
\[ H_\theta(x) = m^{mn} \prod_{i=1}^{m} \prod_{j=0}^{m-1} H_{\theta^{(i)}} \left( \frac{x-i+j+1}{m} \right). \]
\end{lem}

\begin{proof}
Let $\theta \vdash mn$ be an $m$-splittable partition and $(\theta^{(0)}, \ldots, \theta^{(m-1)})$ be its $m$-split. Consider an arbitrary ribbon tableau $T$ of form $\theta$ and content $[m^n]$. We note $T_1, \ldots, T_{m}$ the corresponding ribbon tableaux of form $\theta^{(1)}, \ldots ,\theta^{(m)}$ respectively. In fact, $T_0, \ldots, T_{m-1}$ are all standard Young tableaux, and each cell $w$ of $\theta^{(i)}$ corresponds to a strip $s$ of length $m$ in $T$. Moreover, we can see that the contents of cells in $s$ are exactly $mc(w)-i+1, mc(w)-i+2, \ldots, mc(w)-i+m$. These facts are independent of the choice of $T$. Therefore we have:
\begin{align*}
H_\theta(x) &= \prod_{w \in \theta} (x+c(w)) = \prod_{i=1}^{m} \prod_{w \in \theta^{(i)}} \prod_{j=0}^{m-1} (x+mc(w)-i+j+1) \\
&= m^{mn} \prod_{i=1}^{m} \prod_{j=0}^{m-1} \prod_{w \in \theta^{(i)}} \left( \frac{x-i+j+1}{m}+c(w) \right) \\
&= m^{mn} \prod_{i=}^{m} \prod_{j=0}^{m-1} H_{\theta^{(i)}} \left( \frac{x-i+j+1}{m} \right).
\end{align*}
\end{proof}

As a final remark, the operation in Figure \ref{fig:m-split} is sometimes called drawing the \emph{abacus display} in some literature in algebra. Readers can refer to, for example, Section 2.7 of \cite{james1981representation} for more details on $p$-core, $p$-quotient and abacus display, even though there is only a specialized version of Theorem \ref{thm:character-factorization} in this reference.

\section{Generalization of the quadrangulation relation} \label{sec:app}

In this section, using Theorem \ref{thm:character-factorization}, we establish a relation between $m$-hypermaps and $m$-constellations in arbitrary genus that generalizes the quadrangulation relation. We then recover a result in \cite{chapuy2009asymptotic} on the asymptotic behavior of $m$-hypermaps related to that of $m$-constellations. Finally, by exploiting symmetries of the generating function, we are able to arrange our generalized relation in a form with all coefficients being positive integers, for which there might exist a combinatorial explanation.

\subsection{From series to numbers}
We start by a link between the series $R_H$ and $R_C$, using Theorem \ref{thm:character-factorization}.

\begin{prop} \label{prop:link-of-hypermaps-and-constellation-series}
The generating series $R_H$ and $R_C$ are related by the following equation.
\[ R_H(x,\underline{\mathbf{y}},z) = \prod_{j=1}^{m} R_C \left( \left[ x_i \gets \frac{x-j+i}{m} \right], \left[ y_i \gets \frac{y_i}{m} \right], m^m z \right) \]
\end{prop}
\begin{proof}
We take the expressions of $R_H$ and $R_C$ from Proposition \ref{prop:series-in-big-characters-simplified}. We observe that, in the expression of $R_H$, we only need to consider those $\theta$ that are $m$-splittable according to Theorem \ref{thm:character-factorization}. For such $\theta$, let $(\theta^{(1)}, \ldots, \theta^{(m)})$ be its $m$-split, and we have the following equality derived from Theorem \ref{thm:character-factorization}.

\begin{displaymath}
\chi_{[m^n]}^{\theta} \chi_{m\mu}^{\theta} = n! z_{\mu} \sum_{\mu^{(1)} \uplus \cdots \uplus \mu^{(m)} = \mu} \prod_{i=1}^m \frac{f^{\theta^{(i)}} \chi_{\mu^{(i)}}^{\theta^{(i)}}}{(|\theta^{(i)}|)! z_{\mu^{(i)}}}
\end{displaymath}

We then substitute the equality above and Lemma \ref{lem:from-Hm-back-to-H} into the expression of $R_H$ in Corollary \ref{prop:series-in-big-characters-simplified} to factorize $R_H$ into a product of $R_C$ evaluated on different points as follows:

\begin{align*}
&\quad R_H(x,\underline{\mathbf{y}},z)\\
&= \sum_{n \geq 1} (m^m z)^n \sum_{\mu \vdash n} y_\mu m^{-l(\mu)} \sum_{\substack{\theta \vdash mn \\ \mu^{(1)} \uplus \cdots \uplus \mu^{(m)} = \mu}} \prod_{i=1}^m \left( \frac{f^{\theta^{(i)}} \chi_{\mu^{(i)}}^{\theta^{(i)}}}{(|\theta^{(i)}|)! z_{\mu^{(i)}}}  \prod_{j=0}^{m-1} H_{\theta^{(i)}} \left( \frac{x-i+j+1}{m} \right) \right) \\
&= \prod_{j=1}^{m} R_C \left( \left[ x_i \gets \frac{x-j+i}{m} \right], \left[ y_i \gets \frac{y_i}{m} \right], m^m z \right).
\end{align*}
\end{proof}

This link between $R_H$ and $R_C$ can be translated directly into a link between the series $H(x, \underline{\mathbf{y}}, z, u)$ of $m$-hypermaps and the series $C(\underline{\mathbf{x}}, \underline{\mathbf{y}}, z, u)$ of $m$-constellations, resulting in our main result as follows.

\begin{thm}[Relations of series of constellations and hypermaps] \label{thm:link-between-H-and-C}
The generating series of $m$-constellations and $m$-hypermaps are related by the following formula:
\[ H(x,\underline{\mathbf{y}},z,u) = m \sum_{j=1}^{m} C \left( \left[ x_i \gets \frac{x+(i-j)u}{m} \right], \left[ y_i \gets \frac{y_i}{m} \right], m^{m-1} z, u \right). \]
\end{thm}
\begin{proof}
This comes directly from a substitution of \eqref{eq:H-to-RH} and \eqref{eq:C-to-RC} into Proposition \ref{prop:link-of-hypermaps-and-constellation-series}.
\end{proof}

We note $H^{(g)}(x,\underline{\mathbf{y}},z) = [u^{2g}]H(x,\underline{\mathbf{y}},z,u)$ and $C^{(g)}(\underline{\mathbf{x}},\underline{\mathbf{y}},z) = [u^{2g}]C(\underline{\mathbf{x}},\underline{\mathbf{y}},z,u)$ the generating functions of $m$-hypermaps and $m$-constellations of genus $g$ respectively. We can now express the following corollary concerning the link between $m$-hypermaps and $m$-constellations with respect to the genus.

\begin{coro} \label{coro:link-hypermap-constellation-with-genus}
We have the following relation between the generating series $H^{(g)}$ and $C^{(g)}$:
\[ H^{(g)}(x,\underline{\mathbf{y}},z) = \sum_{k=0}^{g} \frac{m^{2g-2k}}{m(2k)!} \Bigg( \sum_{j=1}^{m} \bigg(\sum_{i=1}^{m} (i-j)\frac{\partial}{\partial x_i} \bigg)^{2k} C^{(g - k)} \Bigg)([x_i \gets x], \underline{\mathbf{y}}, z). \]
\end{coro}
\begin{proof}
We want to compute $H^{(g)}(x,\underline{\mathbf{y}},z) = [u^{2g}]H(x,\underline{\mathbf{y}},z,u)$.
\begin{align*}
&\quad [u^{2g}]H(x,\underline{\mathbf{y}},z,u)  \\
&= m \sum_{j=1}^{m} \sum_{k=0}^{g} [u^{2k}]C^{(g - k)}\left( \left[ x_i \gets \frac{x+(i-j)u}{m} \right], \left[ y_i \gets \frac{y_i}{m} \right], m^{m-1} z \right) \\
&= m \sum_{j=1}^{m} \sum_{k=0}^{g} \frac{1}{(2k)!} \left( \frac{\partial}{\partial u} \right)^{2k} C^{(g - k)}\left( \left[ x_i \gets \frac{x+(i-j)u}{m} \right], \left[ y_i \gets \frac{y_i}{m} \right], m^{m-1} z \right) \bigg|_{u=0} \\
&= m \sum_{j=1}^{m} \sum_{k=0}^{g} \frac{1}{(2k)!} \Bigg( \bigg(\sum_{i=1}^{m} \frac{i-j}{m}\frac{\partial}{\partial x_i} \bigg)^{2k} C^{(g - k)} \Bigg) \left( \left[ x_i \gets \frac{x}{m} \right], \left[ y_i \gets \frac{y_i}{m} \right], m^{m-1} z \right) \\
\end{align*}
To obtain the final result, we then simplify the formula above with the fact that each term in $C^{(g)}$ has the form $x_1^{v_1} \cdots x_m^{v_m} y_{\phi} z^{f_2}$ with $v_1 + \cdots +v_m - mf_2 + |\phi| + f_2 = 2 - 2g$, according to the Euler relation.
\end{proof}

We can further generalize these results. Let $D$ be a subset of $\naturals^{*}$. We define $(m,D)$-hypermaps and $(m,D)$-constellations as $m$-hypermaps and $m$-constellations with the restriction that every hyperface has its degree in $mD$. We note respectively $H_D(x,\underline{\mathbf{y}},z,u)$ and $C_D(\underline{\mathbf{x}},\underline{\mathbf{y}},z,u)$ their generating functions. We now have the following corollary.

\begin{coro}[Main result in the form of series] \label{coro:link-hypermap-constellation-with-genus-extended}
We have the following equations:
\[ H_D(x,\underline{\mathbf{y}},z,u) = m \sum_{j=1}^{m} C_D\left( \left[ x_i \gets \frac{x+(i-j)u}{m} \right], \left[ y_i \gets \frac{y_i}{m} \right], m^{m-1} z, u \right) \]
\[ H_D^{(g)}(x,\underline{\mathbf{y}},z) = \sum_{k=0}^{g} \frac{m^{2g-2k}}{m(2k)!} \Bigg( \sum_{j=1}^{m} \bigg(\sum_{i=1}^{m} (i-j)\frac{\partial}{\partial x_i} \bigg)^{2k} C_D^{(g - k)} \Bigg)([x_i \gets x], \underline{\mathbf{y}}, z). \]
\end{coro}

\begin{proof}
By specifying $y_i=0$ for $i \notin D$ in Corollary \ref{coro:link-hypermap-constellation-with-genus}, we obtain our result.
\end{proof}

By taking $m=2$ and $D=\{ 2 \}, \{ p \}$ or $D$ arbitrary, we recover the quadrangulation relation in \cite{jackson1990character} and its extensions in \cite{jackson1990character-1} and \cite{jackson1999combinatorial} respectively. We define $C^{(g, a_1, \ldots, a_{m-1})}_{n,m,D}$ to be the number of rooted $m$-constellations with $n$ hyperedges, and hyperface degree restricted by the set $D$, with $a_i$ marked vertices of color $i$ for $i$ from $1$ to $m-1$. The number $H^{(g)}_{n,m,D}$ is the counterpart for rooted $m$-hypermaps without markings. These numbers can be easily obtained from corresponding generating series by extracting appropriate coefficients evaluated with $y_i=1$.

According to Theorem 3.1 in \cite{chapuy2009asymptotic}, the number $C^{(g)}_{n,m,D} = C^{(g,0,\ldots,0)}_{n,m,D}$ of $(m,D)$-constellations with $n$ hyperedges without marking grows asymptotically in $ \Theta(n^{\frac{5}{2}(g-1)} \rho_{m,D}^{n})$ when $n$ tends to infinity in multiples of $\operatorname{gcd}(D)$ for some $\rho_{m,D}>0$. Using Corollary \ref{coro:link-hypermap-constellation-with-genus-extended}, we now give a new proof of Theorem 3.2 of \cite{chapuy2009asymptotic} about the asymptotic behavior of the number of $(m,D)$-hypermaps.

\begin{coro}[Asymptotic behavior of $(m,D)$-hypermaps]
For a fixed $g$, when $n$ tends to infinity, we have the following asymptotic behavior of $(m,D)$-hypermaps:
\[ H^{(g)}_{n,m,D} \sim m^{2g} C^{(g)}_{n,m,D}. \]
\end{coro}
\begin{proof}
We observe that, in the second part of Corollary \ref{coro:link-hypermap-constellation-with-genus-extended}, for a fixed $k$, the number of differential operators applied to $C_D^{(g-k)}$ does not depend on $n$, and they are all of order $2k$. Since in an $m$-hypermap, the number of vertices with a fixed color $i$ is bounded by the number of hyperedges $n$, the contribution of the term with $k = t$ is $O(n^{\frac{5}{2}(g-t-1)+2t} \rho_{m,D}^{n})=O(n^{\frac{5}{2}(g-1)-\frac{1}{2}t} \rho_{m,D}^{n})$. The dominant term is therefore given by the case $k=0$, with $C^{(g)}_{n,m,D} = \Theta(n^{\frac{5}{2}(g-1)} \rho_{m,D}^{n})$, and we can easily verify the multiplicative constant.
\end{proof}

Our generalized relation, alongside with its proof, is a refinement of the asymptotic enumerative results established in \cite{chapuy2009asymptotic} on the link between $m$-hypermaps and $m$-constellations.

\subsection{Positivity of coefficients in the expression of $H^{(g)}$} \label{apdx:pos}

In Corollary \ref{coro:link-hypermap-constellation-with-genus}, the generating function $H^{(g)}$ of $m$-hypermaps of genus $g$ is expressed as a sum of generating functions $C^{(g-k)}$ of $m$-constellations with smaller genus applied to various differential operators. It is not obvious that this sum can be arranged into a sum with positive coefficients of all terms, but we will show that it is indeed the case. For $m=3$ and $m=4$, we have the following relations.

\begin{coro}[Generalization of the quadrangulation relation, special case $m=3,4$] \label{coro:counting-relation-m-3-4}
For $m=3,4$, we have
\begin{displaymath}
H^{(g)}_{n,3,D} = \sum_{i=0}^{g} 3^{2g-2i} \sum_{l=0}^{2i} \frac{2 \cdot 2^{l} + (-1)^{l}}{3} C^{(g-i, l, 2i-l)}_{n,3,D},
\end{displaymath}

\begin{displaymath}
H^{(g)}_{n,4,D} = \sum_{i=0}^{g} 4^{2g-2i} \sum_{l_1, l_2 \geq 0, l_1 + l_2 \leq 2i} \frac{2 (3^{l_1}2^{l_2} + 2^{l_2}(-1)^{l_1})}{4} C^{(g-i,l_1,l_2,2i-l_1-l_2)}_{n,4,D}.
\end{displaymath}
\end{coro}

We notice that the coefficients are always positive integers. This is not a coincidence. In fact, by carefully rearranging terms, we can obtain the following relation, whose proof is the subject of this section.

\begin{coro}[Generalization of the quadrangulation relation, for arbitrary $m$] \label{coro:general-counting-relation-in-numbers}
With certain coefficients $c^{(m)}_{k_1, \ldots, k_{m-1}}$ all integral and positive, we have
\begin{displaymath}
H^{(g)}_{n,m,D} = \sum_{i=0}^{g} m^{2g-2i} \sum_{\substack{k_1, \ldots, k_{m-1} \geq 0 \\ k_1 + \cdots + k_{m-1} = 2i}} c^{(m)}_{k_1, \ldots, k_{m-1}} C^{(g-i, k_1, \ldots, k_{m-1})}_{n,m,D}.
\end{displaymath}
\end{coro}

In the following, we will deal with this problem of positivity of coefficients and prove the two corollaries above. We start from the observation that the series $C^{(g)}(\underline{\textbf{x}},\underline{\textbf{y}},z)$ is symmetric in $x_i$. This can be seen algebraically from the expression of $R_C$ in Proposition \ref{prop:series-in-big-characters-simplified}, or bijectively with a ``topological surgery'' that permutes the order of two consecutive colors (details are left to the readers). We can thus deduce the following property of $C^{(g)}$.

\begin{prop} \label{prop:constellation-symmetric}
Let $k_1, k_2, \ldots, k_m$ be natural numbers and $\sigma \in S_m$ an arbitrary permutation. We have the following equality.
\[ \left( \frac{\partial^{k_1 + \cdots + k_m}}{\partial x_1^{k_1} \cdots \partial x_m^{k_m}} C^{(g)} \right)([x_i \gets x],\underline{\textbf{y}},z) = \left( \frac{\partial^{k_1 + \cdots + k_m}}{\partial x_1^{k_{\sigma(1)}} \cdots \partial x_m^{k_{\sigma(m)}}} C^{(g)} \right)([x_i \gets x],\underline{\textbf{y}},z) \]
\end{prop}
\begin{proof}
Since in the evaluation all $x_i$ are given value $x$, any interchange of variables $x_i$ in the series has no effect on the evaluation.
\end{proof}

We now define a new sequence of differential operators $D^{(2k)}$. For $m = 2p$ even, we define
\[D^{(2k)} = 2 \sum_{j=1}^{p} \left( \sum_{i=1}^{m} (i-j)\frac{\partial}{\partial x_i} \right)^{2k}.\]
For $m = 2p + 1$ odd, we define
\[D^{(2k)} = \left( \sum_{i=1}^{m} (i-p-1)\frac{\partial}{\partial x_i} \right)^{2k} + 2\sum_{j=1}^{p} \left( \sum_{i=1}^{m} (i-j)\frac{\partial}{\partial x_i} \right)^{2k}. \]
Using these differential operators, we can rewrite the equations in Corollary \ref{coro:link-hypermap-constellation-with-genus} as follows.

\begin{prop} \label{prop:sum-in-big-D}
We have the following equation.
\[ \left( \sum_{j=1}^{m} \left( \sum_{i=1}^{m} (i-j)\frac{\partial}{\partial x_i} \right)^{2k} C^{(g - k)} \right)([x_i \gets x], \underline{\textbf{y}}, z) = \left( D^{(2k)} C^{(g - k)} \right) ([x_i \gets x], \underline{\textbf{y}}, z) \]
\end{prop}
\begin{proof}
We observe that, according to Proposition \ref{prop:constellation-symmetric}, for any $g$ and $j$,
\begin{align*}
&\quad \left( \left( \sum_{i=1}^{m} (i-j)\frac{\partial}{\partial x_i} \right)^{2k} C^{(g)} \right)([x_i \gets x], \underline{\textbf{y}}, z) \\
&=  \left( \left( \sum_{i=1}^{m} (-(m+1-i) + j)\frac{\partial}{\partial x_i} \right)^{2k} C^{(g)} \right) ([x_i \gets x], \underline{\textbf{y}}, z) \\
&=  \left( \left( \sum_{i=1}^{m} (i - (m+1-j))\frac{\partial}{\partial x_i} \right)^{2k} C^{(g)} \right) ([x_i \gets x], \underline{\textbf{y}}, z).
\end{align*}
With this equality, the proposition is easily verified.
\end{proof}

We define two kinds of coefficients $e^{(m),j}_{k_1, \ldots, k_{m-1}}$ and $d^{(m)}_{k_1, \ldots, k_{m-1}}$ for $k_i \geq 0$ as follows:
\[ e^{(m),j}_{k_1, \ldots, k_{m-1}} = \prod_{1 \leq i \leq m, i \neq j} (i-j)^{k_{i-j \mod m}}, \]
\[ d^{(m)}_{k_1, \ldots, k_{m-1}} = \begin{cases} 2\sum_{j=1}^{p} e^{(m),j}_{k_1, \ldots, k_{m-1}}, &(m=2p) \\ e^{(m),p+1}_{k_1, \ldots, k_{m-1}} + 2\sum_{j=1}^{p} e^{(m),j}_{k_1, \ldots, k_{m-1}}, &(m=2p+1) \end{cases}. \]

We can now rewrite the equation in Proposition \ref{prop:sum-in-big-D} with these coefficients.

\begin{prop} \label{prop:rewrite-with-new-coeff}
For any $m \geq 2$, we have
\begin{align*}
&\quad (D^{(2k)} C^{(g)}) ([x_i \gets x],\underline{\textbf{y}},z) \\
&= \left( \sum_{k_1 + \ldots + k_{m-1} = 2k} \binom{2k}{k_1, \ldots, k_{m-1}} d^{(m)}_{k_1, \ldots, k_{m-1}} \frac{\partial^{2k}}{\partial x_1^{k_1} \cdots \partial x_{m-1}^{k_{m-1}}} C^{(g)} \right) ([x_i \gets x],\underline{\textbf{y}},z).
\end{align*}
\end{prop}
\begin{proof}
For $m=2p$,
\begin{align*}
&\quad (D^{(2k)} C^{(g)}) ([x_i \gets x],\underline{\textbf{y}},z) \\
&= 2 \left( \sum_{j=1}^{p} \left( \sum_{i=1}^{m} (i-j)\frac{\partial}{\partial x_i} \right)^{2k} C^{(g)} \right) ([x_i \gets x],\underline{\textbf{y}},z) \\
&= 2 \left( \sum_{j=1}^{p} \sum_{k_1 + \ldots + k_m = 2k} \binom{2k}{k_1, \ldots, k_m} \left( \prod_{i=1}^{m} (i-j)^{k_{i}} \right) \frac{\partial^{2k}}{\partial x_1^{k_1} \cdots \partial x_m^{k_m}} C^{(g)} \right) ([x_i \gets x],\underline{\textbf{y}},z) \\
&= 2 \left( \sum_{k_1 + \ldots + k_m = 2k} \binom{2k}{k_1, \ldots, k_m} \left( \sum_{j=1}^{p} \prod_{i=1}^{m} (i-j)^{k_{i-j\mod m}} \right) \frac{\partial^{2k}}{\partial x_1^{k_1} \cdots \partial x_m^{k_m}} C^{(g)} \right) ([x_i \gets x],\underline{\textbf{y}},z) \\
&= \left( \sum_{k_1 + \ldots + k_{m-1} = 2k} \binom{2k}{k_1, \ldots, k_{m-1}} d^{(m)}_{k_1, \ldots, k_{m-1}} \frac{\partial^{2k}}{\partial x_1^{k_1} \cdots \partial x_{m-1}^{k_{m-1}}} C^{(g)} \right) ([x_i \gets x],\underline{\textbf{y}},z)
\end{align*}

The computation for $m=2p+1$ is similar.
\end{proof}

We will now show that the coefficients $d^{(m)}_{k_1, \ldots, k_{m-1}}$ are all positive integers divisible by $m$. We start with a lemma concerning $e^{(m),j}_{k_1, \ldots, k_{m-1}}$, which we will use for telescoping.

\begin{lem} \label{lem:telescoping-e}
For $j \leq m/2$, we have $\left| e^{(m),j+1}_{k_1, \ldots, k_{m-1}} \right| \leq \left| e^{(m),j}_{k_1, \ldots, k_{m-1}}\right| $.

Moreover, if $k_{m-j} \geq 1$, we have $\left| e^{(m),j+1}_{k_1, \ldots, k_{m-1}} \right| \leq \frac{j}{m-j} \left| e^{(m),j}_{k_1, \ldots, k_{m-1}} \right| $.
\end{lem}
\begin{proof}
The result is a direct consequence of the following formula coming from the definition of $e^{(m),j}_{k_1, \ldots, k_{m-1}}$:
\[ e^{(m),j+1}_{k_1, \ldots, k_{m-1}} = \left( \frac{-j}{m-j} \right)^{k_{m-j}} e^{(m),j}_{k_1, \ldots, k_{m-1}}. \]
\end{proof}

We can now prove the positivity of $d^{(m)}_{k_1, \ldots, k_{m-1}}$.

\begin{thm} \label{thm:positivity-of-differential-operator-coefficient}
For all $m \geq 2$ and $k_1, \ldots, k_{m-1}$ natural numbers, $d^{(m)}_{k_1, \ldots, k_{m-1}}$ is a positive integer. Moreover, it is always divisible by $m$.
\end{thm}
\begin{proof}
By definition, $d^{(m)}_{k_1, \ldots, k_{m-1}}$ is an integer.

For the case $k_1 = \cdots = k_{m-1} = 0$, we have $d^{(m)}_{0, \ldots, 0} > 0$ by definition. We now suppose that all $k_i$ are not zero. Let $t$ be the largest index such that $k_t > 0$, we then have $k_{t+1} = \cdots = k_{m-1} = 0$. If $t \leq m/2 $, we have $d^{(m)}_{k_1, \ldots, k_{m-1}} > 0$, since all terms of the sum in the definition are positive. We now suppose that $t > p$ for $m=2p$ and $m=2p+1$. We always have $e^{(m),m-t}_{k_1, \ldots, k_{m-1}} \geq t^{k_t} > 0$ in this case.

We start from the case $m=2p$. According to Lemma \ref{lem:telescoping-e}, we have
\begin{align*}
d^{(m)}_{k_1, \ldots, k_{m-1}} &= 2\sum_{j=1}^{p} e^{(m),j}_{k_1, \ldots, k_{m-1}} \\
&\geq 2(m-t)e^{(m),m-t}_{k_1, \ldots, k_{m-1}} - 2(t-p) \left| e^{(m),m-t+1}_{k_1, \ldots, k_{m-1}} \right| \\
&\geq 2(m-t) \left( e^{(m),m-t}_{k_1, \ldots, k_{m-1}} - \left( 1-\frac{p}{t} \right) e^{(m),m-t}_{k_1, \ldots, k_{m-1}} \right) \\
&= \frac{2p(m-t)}{t}e^{(m),m-t}_{k_1, \ldots, k_{m-1}} > 0.
\end{align*}

The computation for the case $m=2p+1$ is similar.
\begin{align*}
d^{(m)}_{k_1, \ldots, k_{m-1}} &= e^{(m),p+1}_{k_1, \ldots, k_{m-1}} + 2\sum_{j=1}^{p} e^{(m),j}_{k_1, \ldots, k_{m-1}} \\
&\geq 2(m-t)e^{(m),m-t}_{k_1, \ldots, k_{m-1}} - (2t-2p+1) \left| e^{(m),m-t+1}_{k_1, \ldots, k_{m-1}} \right| \\
&\geq (m-t) \left( 2e^{(m),m-t}_{k_1, \ldots, k_{m-1}} - \left( 2-\frac{2p-1}{t} \right) e^{(m),m-t}_{k_1, \ldots, k_{m-1}} \right) \\
&= \frac{(2p-1)(m-t)}{t} e^{(m),m-t}_{k_1, \ldots, k_{m-1}} > 0
\end{align*}

The positivity to be proved follows from the computation above.

For divisibility by $m$, we only need to observe that the value of $e^{(m),j}_{k_1, \ldots, k_{m-1}}$ modulo $m$ does not depend on $j$, and that $d^{(m)}_{k_1, \ldots, k_{m-1}}$ is the sum of $m$ such coefficients.
\end{proof}

We define $C^{(g, k_1, \ldots, k_{m-1})}_{n,m,D}(\underline{\textbf{y}})$ as the series of $m$-constellations of genus $g$ with $n$ hyperedges, $k_i$ marked vertices of each color $i$ and $D$ as restriction on degree of hyperfaces. We define similarly $H^{(g)}_{n,m,D}(\underline{\textbf{y}})$. They are the series versions of numbers $C^{(g, k_1, \ldots, k_{m-1})}_{n,m,D}$ and $H^{(g)}_{n,m,D}$, with $y_i$ marking the degree of each hyperfaces.

We can now rewrite Corollary \ref{coro:link-hypermap-constellation-with-genus-extended} into a more agreable form.

\begin{coro} \label{coro:explicit-coeffs}
We have the following relation, with all coefficients positive integers.
\begin{displaymath}
H^{(g)}_{n,m,D}(\underline{\textbf{y}}) = \sum_{i=0}^{g} m^{2g-2i} \sum_{\substack{k_1, \ldots, k_{m-1} \geq 0 \\ k_1 + \cdots + k_{m-1} = 2i}} c^{(m)}_{k_1, \ldots, k_{m-1}} C^{(g-i, k_1, \ldots, k_{m-1})}_{n,m,D}(\underline{\textbf{y}})
\end{displaymath}
Here, $c^{(m)}_{k_1, \ldots, k_{m-1}} = m^{-1} d^{(m)}_{k_1, \ldots, k_{m-1}}$.
\end{coro}

\begin{proof}
Since according to Theorem \ref{thm:positivity-of-differential-operator-coefficient}, $d^{(m)}_{k_1, \ldots, k_{m-1}}$ is a positive integer divisible by $m$, $c^{(m)}_{k_1, \ldots, k_{m-1}}$ is always a positive integer.

The corollary now follows directly from Corollary \ref{coro:link-hypermap-constellation-with-genus} and Proposition \ref{prop:sum-in-big-D}, \ref{prop:rewrite-with-new-coeff}, with the observation that successive derivation of $x_i$ means marking vertices of color $i$ with order, and the restriction imposed by $D$ can be established by specifying $y_i=0$ for any $i \notin D$.
\end{proof}

We have thus proved a generalization of Corollary \ref{coro:general-counting-relation-in-numbers}. The corollary above actually tells us that Corollary \ref{coro:general-counting-relation-in-numbers} holds even if we refine by the degree sequence of hyperfaces. We can obtain Corollary \ref{coro:general-counting-relation-in-numbers} by specifying all $y_i$ to $1$. 

Our result could possibly hint a combinatorial bijection between $m$-hypermaps and some families of $m$-constellations with markings that preserves the degree sequence of hyperfaces, although this might be hindered by the fact that the combinatorial meaning of the coefficient $c^{(m)}_{k_1, \ldots, k_{m-1}}$ itself is not clear. A combinatorial understanding of these coefficients might be indispensable in the quest of the hinted bijection.

\section*{Acknowledgement}

I would like to thank my supervisors, Guillaume Chapuy and Sylvie Corteel from LIAFA, for proofreading and inspiring discussion on both the form and the content of this article.

\bibliographystyle{alpha}
\bibliography{chara-decomp-fang}

\begin{thebibliography}{BDFG04}

\bibitem[BDFG04]{bouttier2004planar}
J.~Bouttier, P.~Di~Francesco, and E.~Guitter.
\newblock Planar maps as labeled mobiles.
\newblock {\em Electron. J. Combin.}, 11(1):Research Paper 69, 27, 2004.

\bibitem[BMS00]{bousquet2000enumeration}
Mireille Bousquet-M{\'e}lou and Gilles Schaeffer.
\newblock Enumeration of planar constellations.
\newblock {\em Adv. in Appl. Math.}, 24(4):337--368, 2000.

\bibitem[Cha09]{chapuy2009asymptotic}
Guillaume Chapuy.
\newblock Asymptotic enumeration of constellations and related families of maps
  on orientable surfaces.
\newblock {\em Combin. Probab. Comput.}, 18(4):477--516, 2009.

\bibitem[JK81]{james1981representation}
Gordon James and Adalbert Kerber.
\newblock {\em The representation theory of the symmetric group}, volume~16 of
  {\em Encyclopedia of Mathematics and its Applications}.
\newblock Addison-Wesley Publishing Co., Reading, Mass., 1981.
\newblock With a foreword by P. M. Cohn, With an introduction by Gilbert de B.
  Robinson.

\bibitem[JV90a]{jackson1990character}
D.~M. Jackson and T.~I. Visentin.
\newblock A character-theoretic approach to embeddings of rooted maps in an
  orientable surface of given genus.
\newblock {\em Trans. Amer. Math. Soc.}, 322(1):343--363, 1990.

\bibitem[JV90b]{jackson1990character-1}
D.~M. Jackson and T.~I. Visentin.
\newblock Character theory and rooted maps in an orientable surface of given
  genus: face-colored maps.
\newblock {\em Trans. Amer. Math. Soc.}, 322(1):365--376, 1990.

\bibitem[JV99]{jackson1999combinatorial}
D.~M. Jackson and T.~I. Visentin.
\newblock A combinatorial relationship between {E}ulerian maps and hypermaps in
  orientable surfaces.
\newblock {\em J. Combin. Theory Ser. A}, 87(1):120--150, 1999.

\bibitem[Lit51]{littlewood1951modular}
D.~E. Littlewood.
\newblock Modular representations of symmetric groups.
\newblock {\em Proc. Roy. Soc. London. Ser. A.}, 209:333--353, 1951.

\bibitem[LZ04]{lando2004graphs}
Sergei~K. Lando and Alexander~K. Zvonkin.
\newblock {\em Graphs on surfaces and their applications}, volume 141 of {\em
  Encyclopaedia of Mathematical Sciences}.
\newblock Springer-Verlag, Berlin, 2004.
\newblock With an appendix by Don B. Zagier, Low-Dimensional Topology, II.

\bibitem[Oko01]{okounkov2001infinite}
Andrei Okounkov.
\newblock Infinite wedge and random partitions.
\newblock {\em Selecta Math. (N.S.)}, 7(1):57--81, 2001.

\bibitem[PS02]{poulalhon2002factorizations}
Dominique Poulalhon and Gilles Schaeffer.
\newblock Factorizations of large cycles in the symmetric group.
\newblock {\em Discrete Math.}, 254(1-3):433--458, 2002.

\bibitem[Ser77]{serre1977linear}
Jean-Pierre Serre.
\newblock {\em Linear representations of finite groups}.
\newblock Springer-Verlag, New York, 1977.
\newblock Translated from the second French edition by Leonard L. Scott,
  Graduate Texts in Mathematics, Vol. 42.

\bibitem[Sta99]{stanley2001enumerative}
Richard~P. Stanley.
\newblock {\em Enumerative combinatorics. {V}ol. 2}, volume~62 of {\em
  Cambridge Studies in Advanced Mathematics}.
\newblock Cambridge University Press, Cambridge, 1999.
\newblock With a foreword by Gian-Carlo Rota and appendix 1 by Sergey Fomin.

\bibitem[VO04]{vershik2004new}
A.~M. Vershik and A.~Yu. Okounkov.
\newblock A new approach to representation theory of symmetric groups. {II}.
\newblock {\em Zap. Nauchn. Sem. S.-Peterburg. Otdel. Mat. Inst. Steklov.
  (POMI)}, 307(Teor. Predst. Din. Sist. Komb. i Algoritm. Metody. 10):57--98,
  281, 2004.

\end{thebibliography}
\end{document}